\documentclass[a4paper,10pt]{article}

\usepackage{amsmath}
\usepackage{amssymb}
\usepackage{a4}

\def \sign{\mathop{\rm sign}\nolimits}

\def \Alt{\mathop{\rm Alt}\nolimits}
\def \dis{\mathop{\rm Dis}\nolimits}
\def \dim{\mathop{\rm dim}\nolimits}
\def \codim{\mathop{\rm codim}\nolimits}

\def \id{\mathop{\rm id}\nolimits}

\def \halt#1#2 {H^{alt} _{#1} (#2)}
\def \haltz#1#2 {H^{alt} _{#1} (#2;\Z ) }

\def \dk#1 {D^k(#1)}
\def \ep#1 {\varepsilon _{#1}}

\def \constant#1 {{\bf \Q }^\bullet _{#1}}
\def \conshft#1#2 {{\bf \Q }^\bullet _{#1} [#2]}

\def \rank{{\mbox{\rm{rank}}}}

\def \MF{\mathop{\rm MF}\nolimits}
\def \orbit{\mathop{\rm Orbit}\nolimits}

\def \C{{\mathbb C}}

\def \Z{{\mathbb Z}}

\def \Q{{\mathbb Q}}

\def \AA{{\cal A}}

\def \KK{{\cal K}}

\def \PP{{\cal P}}

\newtheorem{theorem}{Theorem}[section]
\newtheorem{corollary}[theorem]{Corollary}
\newtheorem{proposition}[theorem]{Proposition}
\newtheorem{lemma}[theorem]{Lemma}
\newtheorem{remark}[theorem]{Remark}
\newtheorem{remarks}[theorem]{Remarks}
\newtheorem{example}[theorem]{Example}
\newtheorem{examples}[theorem]{Examples}
\newtheorem{definition}[theorem]{Definition}
\newtheorem{conjecture}[theorem]{Conjecture}

\newenvironment{proof}
 {\begin{trivlist} \item[\hskip \labelsep {\bf Proof.}]}
 {\hfill$\Box$\end{trivlist}}
\newenvironment{lproof}[1]
 {\begin{trivlist} \item[\hskip \labelsep {\bf Proof (#1).}]}
 {\hfill$\Box$\end{trivlist}}

\begin{document}
\renewcommand{\theenumi}{(\roman{enumi})}
\normalsize

\title{Stratification of Unfoldings of Corank 1 Singularities} 
\author{Kevin Houston \\
School of Mathematics \\ University of Leeds \\ Leeds, LS2 9JT, U.K. \\
e-mail: k.houston@leeds.ac.uk
}
\date{\today }
\maketitle
\begin{abstract}
In the study of equisingularity of families of mappings Gaffney introduced the crucial notion of excellent unfoldings. This definition essentially says that the family can be stratified so that there are no strata of dimension $1$ other than the parameter axis for the family. 
Consider a family of corank $1$ multi-germs with source dimension less than target.
In this paper it is shown how image Milnor numbers can ensure some of the conditions involved in being excellent. The methods used can also be successfully applied to cases where the double point set is a curve. In order to prove the results the rational cohomology description of the disentanglement of a corank $1$ multi-germ is given for the first time. Then, using a simple generalization of the Marar-Mond Theorem on the multiple point space of such maps, this description is applied to give conditions which imply the upper semi-continuity of the image Milnor number. From this the main results follow.

\end{abstract}

\section{Introduction}
An important notion, introduced by Gaffney in \cite{polar}, for studying equisingularity of a family of complex analytic mappings, is that of {\em{excellent unfolding}}, (see Definition~\ref{exc_def}). Essentially, one is ensuring that there are no one-dimensional strata -- except for the parameter axis of the family -- in a stratification of the target of an unfolding of a complex analytic map with isolated instability.

The aim of this paper is show how constancy of the image Milnor number in an unfolding of a corank $1$ finitely $\AA $-determined multi-germ $f:(\C ^n,\underline{z})\to (\C ^p,0)$, with $n<p$, implies some of the conditions in the definition of excellent unfolding.
The image Milnor number for a complex analytic map is analogous to the Milnor number of a hypersurface and is defined to be the number of spheres in a {\em{disentanglement}} of the map, that is, the image of a local stabilisation.

To prove these results we study the topology of multiple point spaces of maps and their stabilisations.
In Section~\ref{mps}, we generalise the theorem of Marar and Mond in \cite{mm} that describes the multiple point spaces of a finitely $\AA$-determined corank $1$ mono-germ, with $n<p$, to the case of multi-germs. 
The proof given here in the mono-germ case is particularly simple and some simplifications are made in the multi-germ case also.

For mono-germs the disentanglement of a map is homotopically equivalent to a wedge of spheres (of possibly varying dimension), see \cite{dm, loctop, vancyc}. In this paper we show that for corank 1 multi-germs with $n<p$ the rational cohomology of the disentanglement is the same as the cohomology of a wedge of spheres (again of possibly varying dimensions). This is done in Section~\ref{dis}. 
We also give a formula that generalises to multi-germs (and corrects) a formula in \cite{gomo} for calculating the rank of the alternating cohomology of the Milnor fibres of the multiple point spaces of the map. This allows us to calculate the image Milnor number in practice since it is just the sum of these ranks.

These descriptions are then applied in Section~\ref{uppersemi} where some simple conditions to ensure upper semi-continuity of the image Milnor number for corank $1$ maps are described. 
 
Thus the image Milnor number behaves like the Milnor number of an isolated complete intersection singularity.

Then in Section~\ref{unfoldings} we show that if the image Milnor number is constant in an unfolding and is upper semi-continuous, then various conditions used in the definition of excellent unfolding hold. For example, Theorem~\ref{0stablesconstant} gives conditions under which the constancy of $\mu _I$ in a family of finitely $\AA $-determined map germs will give constancy of the $0$-dimensional stable singularities in the family (called $0$-stables).
Theorem~\ref{main_thm} is less general and gives conditions under which the constancy of $\mu _I$ in the family gives that the instability locus of the unfolding is the parameter axis.

We also observe some cases where the corank is greater than $1$ and for which similar statements about excellence hold. 

The constancy of $\mu _I$ is important in equisingularity and so the results of Section~\ref{unfoldings} can be applied there.
A significant novelty of the results is that, except for say \cite{diswe}, multi-germs have not been studied much in the equisingularity context.
Furthermore, the results of the paper can be used in the case of mono-germs since multi-germs occur naturally in unfoldings of mono-germs.

We conclude with some remarks and ideas for future work.

The paper was begun while the author was a visitor at Northeastern University. He thanks 
Terry Gaffney and David Massey for their hospitality and discussions on this material.
The author was supported by an EPSRC grant (Reference number EP/D040582/1).

\section{Multiple point spaces}
\label{mps}

The standard definition of the multiple point spaces of a map is the following.
\begin{definition}
Suppose that $f:X\to Y$ is a continuous map of topological spaces. Then 
the {\em{$k$th multiple point space of $f$}}, denoted $D^k(f)$, is the set 
\[
D^k(f):={\rm{closure}} \{ (x_1,\dots ,x_k)\in X^k | f(x_1)=\dots
= f(x_k), {\mbox{ such that }} x_i\neq x_j , i\neq j \} .
\] 
\end{definition}
The group of permutations on $k$ objects, denoted $S_k$, acts on $D^k(f)$
in the obvious way: permutation of copies of $X$ in $X^k$.

This definition does not behave well under deformations. Consider $f(x)=(x^2,x^3)$. We have $D^2(f)=\emptyset $ but for a perturbation $f_t(x)=(x^2,x^3+tx)$ with $t\neq 0$ there will be a non-empty double point set. By restricting our class of maps we can give a definition of multiple point spaces that does behave well under deformations.

\begin{definition}
A map $f:(\C ^n,0)\to (\C ^p,0)$, $n<p$, is called {\em{corank $1$}} if the rank of the differential of $f$ at $0$ is at least $n-1$, i.e., $\rank \, df(0) \geq n-1$. Note that this means that potentially $f$ could be an immersion.

We say that a multi-germ $f:(\C ^n,\underline{z} )\to (\C ^p,0)$, $n<p$, is {\em{corank $1$}} if each branch is corank $1$, 
i.e., $f:(\C ^n,z_i)\to (\C ^p,0)$ is corank $1$ for all $z_i\in \underline{z}$.
\end{definition}

For corank 1 mono-germs there is a good way to define a type of multiple point space by $S_k$-invariant functions
described by Mond in \cite{r2r3} and we recall this now. 

If $G:\C ^{n-1} \times \C \to \C $ is a function in the variables
$x_1,\dots ,x_{n-1}, y$, then define 
$V^k_i(G):\C ^{n+k-1} \to \C $ to be
\[
\begin{array}{|cccccccc|}
1 & y_1 & \dots & y_1^{i-1} & G(x,y_1) & y_1^{i+1} & \dots & 
y_1^{k-1} \\
\vdots & \vdots & & \vdots & \vdots  & \vdots & & \vdots \\ 
1 & y_k & \dots & y_k^{i-1} & G(x,y_k) & y_k^{i+1} & \dots & 
y_k^{k-1}
\end{array}
\Bigg/
\begin{array}{|cccc|}
1 & y_1 & \dots & y_1^{k-1} \\
\vdots  & \vdots & & \vdots \\
1 & y_k & \dots & y_k^{k-1} 
\end{array}
\]
for $i=1,\dots , k-1$.
This function is $S_k$-invariant and, if $G$ is holomorphic, 
then so is $V_i^k(G)$.
Following Marar and Mond in \cite{mm} we can make the following definition.
\begin{definition}
If $f(x_1,\dots ,x_{n-1},y)=(x_1,\dots ,x_{n-1}, f_1(x,y), 
\dots ,f_{p-n+1}(x,y))$, then the {\em{$k$th (unfolding) multiple point space}}, denoted $\widetilde{D}^k(f)$, is generated in $\C ^{n+k-1}$ by 
$V^k_r(f_s)=0$, where $r=1,\dots ,k-1$, and $s=1,\dots ,p-n+1$. 
\end{definition}
We shall drop the reference to unfolding in the definition.
Note that $D^k(f)\subseteq \widetilde{D} ^k(f)$ and that the relation can be strict,
for example let $f(x)=(x^2,x^3)$. The point here is that $\widetilde{D}^2(f)$ will detect that the cusp point has multiplicity $2$ and so will produce a double point when deformed.

For a multi-germ we define $\widetilde{D}^k(f)$ inductively. 
Let $f:(\C ^n, \underline{z} )\to (\C ^p,0)$ be a corank 1 multi-germ where $\underline{z}=\{ z_1, \dots ,z_s\}$. 
For the bi-germ $\{ f_i , f_j\}$ we define 
\[
\widetilde{D}^2(\{ f_i, f_j\})=
\left\{
\begin{array}{ll}
D^2(\{f_i,f_j\}), & {\mbox{ for }} i\neq j \\ 
\widetilde{D}^2(f_i), &  {\mbox{ for }} i=j. 
\end{array}
\right.
\]
Then $\widetilde{D}^2(f)$ is the union of $\widetilde{D}^2(\{ f_i, f_j\})$ for all $i$ and $j$. We can (and do) identify this space with a subset of $(\C ^n, \underline{z}) \times (\C ^n, \underline{z})$ as in \cite{gomo}.

Define $\widetilde{D} ^k(f)= \widetilde{D}^2(\varepsilon ^{k-1,k-2})$ where $\varepsilon ^{k-1,k-2}$ is the map 
$\varepsilon ^{k-1,k-2}:\widetilde{D}^{k-1}(f)\to \widetilde{D}^{k-2}(f)$ given by $\varepsilon ^{k-1,k-2}(x_1,\dots ,x_{k-2}, x_{k-1})=(x_1,\dots ,x_{k-2})$. We can identify the resulting space with a subset of $(\C ^n,\underline{z} )^k$.

In \cite{mm} the authors show (roughly speaking) that a mono-germ is finitely $\AA $-determined if and only if the multiple point spaces are isolated complete intersection singularities. They also show that the map is stable if and only if the  spaces are non-singular. We now show that the same is true for multi-germs. It is possible to synthesise a proof of this from \cite{mm} (because it is practically proved there but not stated). Instead we opt to give an alternative method that is simpler in a number of respects.

\begin{theorem}[Cf.\ \cite{mm}]
\label{mps_for_stable}
Suppose that $F:(\C ^n,\underline{z})\to (\C ^p,0)$, $n<p$, is a finite multi-germ of corank 1.
Then, $F$ is stable if and only if, for all $k$, $\widetilde{D}^k(F)$ is either non-singular of dimension $nk-p(k-1)$ or empty.
\end{theorem}
Before giving the proof we will state some important corollaries.

It is possible to show that for stable maps that $D^k(f)=\widetilde{D}^k(f)$. 
This allows us to describe the (unfolding) multiple point spaces for finitely $\AA $-determined maps in a manner used  for mono-germs by Gaffney in \cite{gaffneymps}. 
\begin{proposition}
Suppose that $f:(\C ^n,\underline{z})\to (\C ^p,0)$ is a finitely $\AA $-determined multi-germ with
versal unfolding $F:(\C ^n\times \C ^q,\underline{z} \times \{ 0\} )\to 
(\C ^p \times \C ^q, 0\times 0 )$ given by $F(x,\lambda )=(f_\lambda (x), \lambda )$,
then the {\em{$k$th multiple point space}} is defined to be
$\widetilde{D} ^k(f):=D^k(F) \cap \{ \lambda _1 = \dots = \lambda _k=0 \}$,
where $\lambda _i$, $i=1, \dots ,k$, are the  
unfolding coordinates 
in $(\C ^{n+q},\underline{z} \times \{ 0 \} )^k$. 
\end{proposition}
In fact we could have taken this as a definition of multiple point space for a finitely $\AA $-determined map-germ (mono or multi). Such a definition is independent of the unfolding used in the sense that different unfoldings
produce multiple point spaces homeomorphic by $S_k$-equivariant maps.

\begin{corollary}[Cf.\ \cite{mm}]
\label{mps_for_findet}
Suppose that $F:(\C ^n,\underline{z})\to (\C ^p,0)$, $n<p$, is a finite multi-germ of corank 1.
Then, $F$ is finitely $\AA $-determined at $0\in (\C^ p,0)$ if and only if
for each $k$ with $nk-p(k-1)\geq 0$, $\widetilde{D}^k(F)$ is either an isolated complete
intersection singularity of dimension $nk-p(k-1)$ or empty, and if furthermore, for those $k$
such that $nk-p(k-1)<0$, $\widetilde{D}^k(F)$ consists of at most points given by a $k$-tuple of 
elements from $\underline{z}$.
\end{corollary}
\begin{proof}
For this we use Gaffney's Geometric Criteria for finite determinacy, see \cite{wall}:
The map $F$ is finitely determined if and only if it has an isolated instability at the point
$0\in (\C ^p,0)$. 

From Theorem \ref{mps_for_stable} we know that, except for $k$-tuples from $\underline{z}$, the multiple point spaces are non-singular of dimension $nk-p(k-1)$ if and only if the map is stable. Thus, if $F$ is finitely determined, the multiple point spaces have isolated singular points. It is elementary to calculate that the multiple point spaces are the right dimensions to be complete intersections.
Conversely, if the multiple point spaces have isolated singular points and are complete intersections, then, again by Theorem~\ref{mps_for_stable}, $F$ must be stable outside the origin. Hence, $F$ is finitely determined.  
\end{proof}
Note that the standard definition of isolated complete intersection singularity includes the case where the singularity is in fact non-singular.

Recall the following from \cite{gomo} Section 3.
Let $\PP = (k_1, \dots , k_1, k_2, \dots , k_2, \dots , k_r, \dots , k_r)$ be a partition of $k$ with $1\leq k_1 <k_2 < \dots < k_r$ and $k_i$ appearing $\alpha _i$ times (so $\sum \alpha _i k_i = k$). For every $\sigma \in S_k$ there is a partition $\PP _\sigma $, that is, in its cycle decomposition there are $\alpha _i$ cycles of length $k_i$. For any partition we can associate any one of a number of elements of $S_k$.

\begin{definition}[\cite{mm}]
For a map $f$ we define $\widetilde{D}^k(f,\PP _\sigma )$ to be the restriction of $\widetilde{D}^k(f)$ to the fixed point set of $\sigma $. This is well-defined for a partition, i.e., if $\PP _{\sigma } = \PP _{\sigma '}$, then 
$\widetilde{D}^k(f,\PP _{\sigma })$ is equivalent as an $S_k$-invariant isolated complete intersection
singularity to 
$\widetilde{D}^k(f,\PP _{\sigma '})$.
\end{definition}

\begin{corollary}[Cf.\ \cite{mm}]
Let $F:(\C ^n,\underline{z})\to (\C ^p,0)$, $n<p$, be a corank $1$ finite complex analytic multi-germ and $\PP $ be a partition of $k$. 
\begin{enumerate}
\item If $F$ is stable, then $\widetilde{D}^k(F, \PP )$ is non-singular of  dimension $nk-p(k-1)+ \sum _i \alpha _i -k$.
\item If $F$ is finitely $\AA $-determined, then two cases occur.
	\begin{enumerate}
	\item If $nk-p(k-1)+\sum _i \alpha _i -k\geq 0$, then $\widetilde{D}^k(F, \PP )$ is an isolated complete intersection
singularity of dimension $nk-p(k-1)+\sum _i \alpha _i -k$.
	\item If $nk-p(k-1)+\sum _i \alpha _i -k<0$, then $\widetilde{D}^k(F, \PP )$ is at most an isolated point consisting of a $k$-tuple of the elements of $\underline{z}$.
	\end{enumerate}
\end{enumerate}
\end{corollary}
\begin{proof}
Part (i) follows from Theorem~\ref{mps_for_stable} and the fact that $\widetilde{D}^k(F,\PP )$ is defined in 
$\widetilde{D}^k(F)$ by the right number of linear equations.
Part (ii) is a corollary of (i) using the same ideas as in the proof of Corollary~\ref{mps_for_findet}.
\end{proof}

\begin{lproof}{of Theorem \ref{mps_for_stable}}
Consider first the case of mono-germs.
The normal form for a stable corank 1 map was given by Morin, see \cite{morin}. 
The result is that  if $F$ is stable, then $F$ is $\AA $-equivalent to a germ given by
\[
(u,w,y^l+ \sum_{i=1}^{l-2} u_iy^i, \sum_{i=1}^{l-1} w_{1i}y^i, \dots , \sum_{i=1}^{l-1} w_{ri}y^i ) ,
\]
where $r=p-n$, $u=(u_1,\dots ,u_{l-2})$ and $w=(w_{0,0}, \dots , w_{r,l-1},w_1,\dots w_q)$. (The 
$(w_1,\dots w_q)$ arise from trivially unfolding.)  
Conversely, any map-germ which has $l$-jet equal to the above is stable.

If we assume that $F$ is stable, then it is elementary to show, using the normal form above and the Vandermonde description of the defining equations  that the multiple point spaces are non-singular.

Now, for the converse, suppose that $F$ is such that 
$\widetilde{D}^k(F)$ is non-singular of dimension $nk-p(k-1)$ or empty for all $k$. 
Then we show that  $F$ is $\AA $-equivalent to a germ which has $l$-jet equal to that given above.

Assume that $F$ has multiplicity $l$. We can make standard left and right changes of coordinates and a Tschirnhaus transformation to show that $F$ is $\AA $-equivalent to a multi-germ with $l$-jet equal to 
\[
(u,w, y^l + \sum_{i=1}^{l-2} f_i(u,w)y^i , 
 \sum_{i=1}^{l-1} h_{1,i}(u,w)y^i , \dots ,
 \sum_{i=1}^{l-1} h_{r,i}(u,w)y^i ) ,
\]
for some functions $f_i$ and $h_{i,j}$, with zero constant term. 

In effect we will use induction to show $F$ has the required form.
Consider first $\widetilde{D}^2(F)$. 
The defining equations
for $\widetilde{D}^2(F)$ have linear terms arising from the linear terms in 
$f_1$ and $h_{j,1}$ for $j=1,\dots ,r$.
To ensure the non-singularity of $\widetilde{D}^2(F)$
the $p-n+1$ defining equations should have a non-zero linear term. 
There exists a change
of coordinates involving only $u$ and $w$ so that (without loss
of generality)
$f_1(u,w)=u_1$ and $h_{i,1}=w_{i,1}$ for $i=1,\dots ,r$.
Then through a left change of coordinates we can restore the
components of $F$ to $(u,w, \dots )$.

When we look at $\widetilde{D}^3(F)$ we get $p-n+1$ equations with 
non-zero linear terms arising from the linear terms in $f_1$ and $h_{j,1}$, (i.e.,
the $u_1$, $w_{i,1}$ we found above). The other $p-n+1$ defining
equations have linear terms arising from the linear terms of $f_2$, 
and $h_{i,2}$ and to ensure non-singularity of the multiple point space
we have, after a change of
coordinates,
$f_2(u,w)=u_2$ and $h_{i,2}(u,w)=w_{i,2}$ for $i=1,\dots ,r$.

Thus, by proceeding in this way for all $k<l$ we get a map
\[
(u,w, y^l + \sum_{i=1}^{l-2} u_iy^i , 
 h_{1,l-1}(u,w)y^{l-1} +\sum_{i=1}^{l-2} w_{1,i}y^i , 
\dots ) .
\]

Now, $\widetilde{D}^l(F)$ is also non-singular and so noting that  
 $V_{l}^{l-1}(y^l)$ is linear in $y_j$ we
can deduce that the coefficient of the $y^{l-1}$ term
in each of the last $r-1$ components is linear (in $w$).
Thus, we have produced a map of the statement of the lemma.

We now tackle the multi-germ case.
First suppose that $F$ is stable. 
If the branches of $F$ have multiplicity $(k_1,k_2,\dots , k_s)$, then a multi-germ $\AA $-equivalent to a trivial unfolding of $F$ appears in the stable mono-germ of multiplicity greater than or equal to $\sum _{i=1}^s k_i$. The multiple point spaces of this mono-germ are non-singular by the above working and at the multi-germ in question, a local multiple point spaces is just the trivial unfolding of the multiple point spaces for $F$ we deduce that the multiple point spaces for $F$ are non-singular.

Conversely, let us suppose that the multiple point spaces of $F$ are non-singular. This implies that the multiple point spaces of the branches are non-singular and since these are mono-germs, the branches must be stable. We can now work in the same way as in the proof of Proposition~2.13 of \cite{mm}, i.e., we show that $F\times \dots \times F \to (\C ^p)^s$ is transverse to the big diagonal in  $(\C ^p)^s$, since Mather's Criterion for stable multi-germs will show that $F$ is stable.
\end{lproof}

\section{Disentanglements}
\label{dis}
The disentanglement of a complex analytic map $f$ is the discriminant of a local stabilisation of $f$ and is
an important analytic invariant of $f$ analogous to the Milnor fibre of a hypersurface singularity. In the case of mono-germs where $n\geq p-1$ this topological  space is homotopically equivalent to a wedge of  spheres all of the same dimension, see \cite{dm} and \cite{vancyc}. In the case $n<p-1$, the spheres may be of different dimensions, see \cite{loctop}. 

For multi-germs it is not too difficult to generalise the proofs in \cite{dm} and \cite{vancyc} to show that, for $n\ge p-1$, the disentanglement is again homotopically a wedge of spheres, see \cite{highmult}. 
In the $n<p-1$ multi-germ case the problem is still open. 
However, when the multiple point spaces have a simple structure, for example, they are isolated complete intersection singularities, we are able to show in the next theorem that the rational cohomology of the disentanglement is the same as the cohomology of a wedge of  spheres (of possibly varying dimension).

Suppose that $f:(\C ^n,\underline{z})\to (\C ^p,0)$, $n<p$, is a finitely $\AA $-determined
map-germ with $F:(\C ^n\times \C ^r,\underline{z}\times 0)\to (\C ^p\times \C^r,0)$ a versal unfolding, so that 
$F$ has the form $F(x,t)=(f_t(x),t)$. 
Let $\Sigma $ be the bifurcation space in the unfolding parameter space $\C ^r $, i.e., points such that the map $f_t:\C ^n \to \C ^p $ is unstable.
For corank 1 maps and maps in the nice dimensions the set $\Sigma $ is a proper subvariety of $\C ^r$, see \cite{mapfib}, and so $\C^r \backslash \Sigma $ is connected. 

Given a Whitney stratification of the image of $f$ in $\C ^p $ there exists $\epsilon _0>0$ such that for all $0<\epsilon \leq \epsilon _0$, the real $(2p-1)$-sphere centred at $0$ of radius $\epsilon $ is transverse to the strata of the image of $f$.

Consider the map $f_t|f_t^{-1}(B_\epsilon )$ where $B _\epsilon $ is the closed ball of radius $\epsilon $ centred at $0$ with $\epsilon \leq \epsilon _0$ and $t\in \C^r \backslash \Sigma $. This stable map is called {\em{the disentanglement map}} of $f$ and its image 
is the {\em{disentanglement}} of $f$, denoted $\dis (f)$. This is independent of sufficiently small $\epsilon $ and $t$.
To ease notation we will write $\widetilde{f}$ rather than $f_t|f_t^{-1}(B_\epsilon )$. As $f^{-1}(0)$ is a finite set we can assume that $\widetilde{f}$ has the form $\widetilde{f}:\coprod_{j=1}^{|\underline{z}|}  \widetilde{U}_j \to \C ^p$ where $\widetilde{U}_j$ is  a contractible open set.

We can now generalise Theorem 2.6 of \cite{gomo} which only considered mono-germs.
Let $d(f)=\sup \{ k \,  | \, \widetilde{D}^k(\widetilde{f}) \neq \emptyset \}$ and let $s(f)$ be the number of branches of $f$ through the origin $0$ in $(\C ^p,0)$, i.e., the cardinality of $\underline{z}$.
\begin{theorem}
\label{dis_cohomology}
Suppose that $f:(\C ^n,\underline{z})\to (\C ^p,0)$, $n<p$, is a corank 1 finitely $\AA $-determined
multi-germ. Then, $\widetilde{H}^*(\dis (f);\Q )$ has non-trivial
groups possible only in dimensions $p-(p-n-1)k-1$ for all $2\leq k \leq d(f)$ and if $s(f)>d(f)$, then in
$\widetilde{H}^{d-1}(\dis (f);\Q )$.
\end{theorem}
A simple corollary of this for the case $p=n+1$ is that the only non-trivial groups 
occur in dimension $p-1$. As stated earlier this is known for mono-germs since the disentanglement is homotopically equivalent to a wedge of spheres, see \cite{vancyc}.

To calculate the rational cohomology of the disentanglement we shall use the multi-germ version of Marar and Mond's description of multiple point spaces (Theorem \ref{mps_for_findet}) and
the Image Computing Spectral Sequence, see \cite{gomo, icss}, a relative version of which we now describe.

We will calculate cohomology groups over the rational numbers $\Q $ and
the group $S_k$ will act on $H^i(\widetilde{D}^k(f);\Q)$ in the natural way -- i.e., arising from the action on $\widetilde{D}^k(f)$. We can state and prove the results in this section for integer homology by using the more complicated techniques developed in \cite{go,loctop,icss}, but it is enough for the purposes of excellent unfoldings to use the simpler case of rational cohomology.
  
Define the idempotent operator $\Alt _k$ by
\[
\Alt _k = \frac{1}{k!} \sum_{\sigma \in S_k} \sign (\sigma) \sigma . 
\]
When we apply this operator to $H^i(\widetilde{D}^k(f);\Q)$ we produce what is called the {\em{alternating cohomology}} of $\widetilde{D}^k(f)$. The cohomology of the image of a map may be calculated from the alternating cohomology of its multiple point spaces:
\begin{theorem}[Cf.\ Proposition 2.3 of \cite{gomo}]
Let $f:X\to Y$ be a finite and proper subanalytic map and let $Z$ be a (possibly empty) subanalytic subset of $X$ such that $f|Z$ is also proper. Then, there exists a spectral sequence
\[
E_1^{r,q}= \Alt _{r+1} H^q(D^{r+1}(f), D^{r+1}(f|Z);\Q ) \Rightarrow
H^*(f(X),f(Z);\Q ) ,
\]
where the differential is induced from the natural map 
$\epsilon _{r+1,r}:D^{r+1}(f)\to D^r(f)$ given by 
$\epsilon _{r+1,r} (x_1, x_2, \dots , x_r, x_{r+1}) = (x_1, x_2, \dots , x_r)$.
\end{theorem}
A proof is found in \cite{icss} and can also be generalized from the one in \cite{gomo}. The precise details of the differential will not be required as our spectral sequences will be sparse. Alternatively the differential can be calculated from the terms of the sequence. 
A spectral sequence for a single map rather than a pair also exists if we take $Z=\emptyset $.

The disentanglement is the image of a finite and proper map and hence we can use the Image Computing Spectral Sequence to calculate the cohomology of the disentanglement. 
In \cite{gomo} it is proved that the sequence for the disentanglement map of a corank 1 finitely $\AA $-determined {\em{mono}}-germ collapses at $E_1$. Using the proof there and techniques from \cite{loctop}  it is possible to show that the sequence collapses at $E_2$ for the disentanglement map of a multi-germ.
However, we simplify matters by considering the spectral sequence for a pair of maps: the versal unfolding and the  disentanglement map -- the latter is a restriction of the former. 
As we shall see, the spectral sequence collapses at $E_2$ for this pair also.

Let $F:U\to W$ be the versal unfolding of $f$ such that $\widetilde{f}$ is a restriction of $F$ that gives the disentanglement.
\begin{lemma}
\label{E1_for_F}
The $E_1$ terms of the spectral sequence of $F$ are
\[
E_1^{r,q}(F) \cong 
\left\{ 
\begin{array}{ll}
\Q ^{\binom{s(f)}{r+1} }, & {\mbox{ for }} q= 0 {\mbox{ and }} 1\leq r+1 \leq s(f) , \\
0, & {\mbox{otherwise.}}
\end{array}
\right.
\]
The sequence collapses at $E_2$ and 
\[
E_{\infty }^{r,q} (F) \cong E_2^{r,q} (F) \cong
\left\{ 
\begin{array}{ll}
\Q & {\mbox{ for }} (r,q)= (0,0)  , \\
0 & {\mbox{ for }} (r,q)\neq (0,0).
\end{array}
\right.
\]

\end{lemma}
\begin{proof}
Since $F$ is a complex analytic map the maps $\epsilon _k:D^k(F) \to W$ given by $\epsilon _k(x_1,\dots , x_k)=f(x_1)$ are also complex analytic. As
these maps are proper and finite, their images are complex analytic subsets of $W$. Hence, 
we can Whitney stratify the image of $F$ so that for all $k$ the image of $D^k(F)$ under $\epsilon _k$ is
a union of strata.

Now, by definition of the disentanglement, the image of $F$ is contractible to the space $\{ 0 \}$. This comes from the conic structure theorem for Whitney stratified spaces. The retraction resulting from the conic structure theorem can be achieved via a stratified vector field on $W$. This can be lifted to $D^k(F)$ via $\epsilon _k$ and hence $D^k(F)$ retracts onto
$D^k(F) \cap (F^{-1}(0))^k$. Any point in this latter set is a $k$-tuple chosen from $\underline{z}=(z_1,z_2, \dots ,z_s)$. 
The $k$th multiple point space for the map $F' :\{ z_1, z_2, \dots , z_s \} \to \{ 0 \} $ given by restriction of $F$ will be a subset of $D^k(F) \cap (F^{-1}(0))^k$. We show that they have the same alternating cohomology.

First note that if a $k$-tuple from $\underline{z}$ has any repetitions, then the alternating cohomology of the orbit of this
point under the $S_k$-action is trivial (as at least one simple permutation fixes it). 

Hence if $k\leq s$ and the $k$-tuple has no repetitions,
then the $k$-tuple gives a point in $D^k(F) \cap (F^{-1}(0))^k$ and the alternating 
cohomology of its orbit is $\Q $. If $k\geq s$, then the $k$-tuple has repetitions and
so its orbit does not have any alternating cohomology.

Therefore, the alternating cohomology of $D^k(F) \cap (F^{-1}(0))^k$ and $D^k(F')$
are the same, see Theorem~2.3 and Example~2.5 of \cite{loctop}. Hence, the $E_1$ page of the spectral sequence has the form stated.
This page is depicted in Figure~\ref{E1fig_for_F}.
\begin{figure}
\begin{center}
\begin{tabular}{c|c|c|c|c|c|c|c}
$H^{alt}_2$ & 0 & 0 & 0 & \dots & 0 & 0 & 0  \\
$H^{alt}_1$ & 0 & 0 & 0 & \dots & 0 & 0 & 0  \\
$H^{alt}_0$ & $\Q ^s$ & 
$\Q ^{\binom{s}{2}}$ &
$\Q^{\binom{s}{3}}$ & 
\dots  & $\Q ^s$ & $\Q $ & 0 \\
\hline 
 & $U$ & $D^2(F)$ & $D^3(F)$ & \dots & $D^{s-1}(F)$ & $D^s(F)$ & $D^{s+1}(F)$ 
\end{tabular}
\end{center} 
\caption{\label{E1fig_for_F}$E_1$ page for the map $F$}
\end{figure}

Given the form of $E_1^{*,*}$ the sequence must collapse at the $E_2$ page as all the 
differentials are trivial from this page onwards.

Since $E_*^{r,q}(F)$ converges to $H^*(F(U);\Q ) \cong H^0(\{ 0 \} ;\Q ) \cong \Q $ then 
$E_2^{r,q}$ has the form stated.
\end{proof}

Suppose that $f:(\C ^n,\underline{z})\to (\C ^p,0)$, $n<p$, is a finitely $\AA $-determined multi-germ with versal unfolding $F$ and disentanglement map $\widetilde{f}$. 
Since the disentanglement map of a finitely $\AA $-determined map-germ is stable, $\dim \widetilde{D}^{k}(\widetilde{f})=nk-p(k-1)$ by Theorem~\ref{mps_for_stable}. Hence, if $k> \frac{p}{p-n}$, then $\dim \widetilde{D}^{k}(\widetilde{f})=\emptyset $.

We now describe the $E_1$ page of the spectral sequence for finitely $\AA $-determined multi-germs. Note how sparse it is.
\begin{lemma}
\label{E1sparse}
Suppose that $f:(\C ^n,\underline{z})\to (\C ^p,0)$, $n<p$, is a corank $1$ finitely $\AA $-determined multi-germ, with versal unfolding $F$ and disentanglement map $\widetilde{f}$. 
As before let $d(f)=\sup \{ k \,  | \, \widetilde{D}^k(\widetilde{f}) \neq \emptyset \}$ and $s(f)$ be the number of branches of $f$.

Then,
\[
E_1^{r,q} =
\left\{
\begin{array}{ll}
\Alt _{r+1} H^{\dim \widetilde{D}^{r+1}(\widetilde{f})+1}(D^{r+1}(F), \widetilde{D}^{r+1}(\widetilde{f});\Q ) , & 
{\mbox{for }} q=\dim \widetilde{D}^{r+1}(\widetilde{f})+1\geq 0 \\
 & {\mbox{ and }}r+1\leq d(f),\\
\Q ^{\binom{s(f)}{r+1}} , & q=0 {\mbox{ and }} r+1>d(f), \\
0, & {\mbox{otherwise.}} 
\end{array}
\right.
\]
Here we define $\dim \emptyset =-1$. 
\end{lemma}
\begin{remark}
A schematic version of the $E_1$ page is given in Figure \ref{E1fig_for_pair} for $n=10$ and $p=13$.
We let $s=s(f)$ be the number of branches of $f$.
\begin{figure}
\begin{center}
\begin{tabular}{c|c|c|c|c|c|c|c}
$H^{alt}_9$ & 0 & 0 & 0 & 0 & 0 & 0 & 0  \\
$H^{alt}_8$ & 0 & $\Q ^{r_1}$ & 0 & 0 & 0 & 0 & 0  \\
$H^{alt}_7$ & 0 & 0 & 0 & 0 & 0 & 0 & 0  \\
$H^{alt}_6$ & 0 & 0 & 0 & 0& 0 & 0 & 0  \\
$H^{alt}_5$ & 0 & 0 &  $\Q ^{r_2}$ & 0 & 0 & 0 & 0  \\
$H^{alt}_4$ & 0 & 0 & 0 & 0 & 0 & 0 & 0  \\
$H^{alt}_3$ & 0 & 0 & 0 & 0 & 0 & 0 & 0  \\
$H^{alt}_1$ & 0 & 0 & 0 & $\Q ^{r_3}$ & 0 &  0 & 0  \\
$H^{alt}_2$ & 0 & 0 & 0 & 0 & 0 & 0 & 0  \\
$H^{alt}_0$ & 0 & 0 & 0 & 0 & $\Q ^{\binom{s}{5}} $ & $\Q ^{\binom{s}{6}} $ & $\Q ^{\binom{s}{7}} $  \\
\hline 
 & $D^1$ & $D^2$ & $D^3$ & $D^4$  & $D^5$ & $D^6$ & $D^7$ 
\end{tabular}
\end{center} 
\caption{\label{E1fig_for_pair}$E_1$ page for the pair $(F,\widetilde{f})$}
\end{figure}
\end{remark}
\begin{proof}
When $\widetilde{D}^k(\widetilde{f})\neq \emptyset $  
we have that $D^k(F)$ is non-singular (Theorem~\ref{mps_for_stable}) and forms a disjoint union of Milnor balls for the isolated complete intersection singularities defined by $\widetilde{D}^k(f)$, (Corollary~\ref{mps_for_findet}). Since $\widetilde{f}$ is stable the components of $\widetilde{D}^k(\widetilde{f})$ will be the Milnor fibres for these singularities.
This implies that $H^*(D^k(F),\widetilde{D}^k(\widetilde{f});\Q )$ is non-trivial only in dimension $\dim \widetilde{D}^k(\widetilde{f}) +1= nk-p(k-1) +1 $. Obviously the same is true when we apply the $\Alt _k$ functor.

Hence, for $k$ with $\widetilde{D}^k(\widetilde{f})\neq \emptyset $, the $E_1^{k+1,q}$ terms of the spectral sequence have the desired form. 

When  $\widetilde{D}^k(\widetilde{f})= \emptyset $ we have 
$\Alt _k H^*(D^k(F),\widetilde{D}^k(\widetilde{f});\Q ) \cong \Alt _k H^*(D^k(F);\Q )$. 
These latter groups are described in Lemma~\ref{E1_for_F}.

Hence, the $E_1$ terms of the spectral sequence are as described in the statement of the lemma.
\end{proof}

\begin{lproof}{of Theorem \ref{dis_cohomology}}
The spectral sequence in the previous lemma collapses at $E_2$: The only non-trivial differentials at $E_1$ are those on the
bottom row and as $E_1^{r,0}$ comes from the exact sequence in Lemma~\ref{E1_for_F} for $r+1>d(f)$. Hence, if $s(f)>d(f)$, then
$E_1^{r,0}$ is non-zero for $r+1>d(f)$.

The sequence will now collapse at $E_2$ since the differentials are all equal to the zero map. Thus we read off the cohomology of the pair $(F(U), \widetilde{f}(\widetilde{U}))$ along the diagonals in the usual way. The only non-trivial groups $H^i(F(U),\widetilde{f} (\widetilde{U});\Q )$ are when 
\begin{eqnarray*}
i&=&r+\dim \widetilde{D}^{r+1}(\widetilde{f})+1 \\
&=& r + n(r+1)-p(r+1-1)+1 \\
&=&p-(p-n-1)(r+1)
\end{eqnarray*}
for $2\leq r+1\leq d(f)$.

Thus,  $H^i(F(U), \widetilde{f}(\widetilde{U});\Q)\neq 0 $ only for $i=p-(p-n-1)(r+1)$ where 
$2\leq r+1 \leq d(f)$. 

The only other possible non-trivial entry is $E_2^{d(f)+1,0}$ and this accounts for all $E^{r,q}$ for $r+1>d(f)$.

We know that the image of $F$ is contractible, thus 
\[
H^i(F(U), \widetilde{f}(\widetilde{U});\Q ) \cong \widetilde{H}^{i-1}(\widetilde{f}(\widetilde{U});\Q )
= \widetilde{H}^{i-1}(\dis (f) ;\Q ) .
\]
From this the theorem follows.
\end{lproof}

\begin{remark}
The theorem holds in greater generality since essentially we have only used the idea that the multiple point spaces of $f$ have smoothings for which we can calculate the cohomology.
Further exploration of this is made in Theorem \ref{closingtheorem}.
\end{remark}

\begin{example}
\label{quadpt}
Suppose that $f:(\C ^2, \{ x_1, x_2, x_3, x_4\} )\to (\C ^3,0)$ defines an ordinary quadruple point, i.e., the transverse meeting, except at $0$, of four $2$-planes. Then, $f$ is finitely $\AA $-determined and a disentanglement can be given by just moving one of the planes slightly. That is, the quadruple point is `split' into four triple points. Here we have $s(f)=4$ but 
$\widetilde{D}^4(\widetilde{f}))=\emptyset $ so $s(f)>d(f)$.
\end{example}
\begin{remark}
For the ordinary quadruple point in Example \ref{quadpt} the non-trivial alternating cohomology group arises from the alternating cohomology of the versal unfolding as the $4$th multiple point space of the disentanglement is empty.  
\end{remark}

Each column of the $E_1$ page of the spectral sequence has only one non-trivial term and we will give this a special name.
\begin{definition}
The {\em{$k$th alternating Milnor number}} of $f$ at $y\in \C ^p$, denoted $\mu ^{alt}_k (f,y)$, is 
defined by 
\[
\mu ^{alt}_k (f,y)=
\left\{ 
\begin{array}{ll}
\Alt _k H^*(D^k(F), \widetilde{D}^k(\widetilde{f}) ;\Q ), & {\mbox{ if }} k\leq d(f), \\
\left| \sum_{l=d(f)+1}^{s(f)} (-1)^l \binom{s(f)}{l} \right|, & {\mbox{ if }} k= d(f)+1 {\mbox{ and }} s(f)>d(f) ,\\
0, & {\mbox{otherwise.}}
\end{array}
\right.
\] 
\end{definition}  

\begin{remark}
\label{topmufiszero}
Note that if $s(f)\leq d(f)$, then $\mu _{d(f)+1}^{alt} (f)=0$. Note also that Example~\ref{quadpt} shows that $\mu _{d(f)+1}^{alt} (f)$ may be non-zero. 
\end{remark}

We can now add the $\mu ^{alt}_k$ together to get an analytic invariant for the germ of $f$ at $y$.
\begin{definition}
The {\em{image Milnor number}} of $f$ at $y$, denoted $\mu _I(f,y)$, is the sum of the ranks of the non-zero homology groups of the disentanglement. That is,
\[
\mu _I(f,y) = \sum _k \mu ^{alt}_k (f,y).
\]
If $y=0$, then we will use the notation $\mu _I(f)$.
\end{definition}

\begin{remarks}
\begin{enumerate}
\item The image Milnor number was first defined by Mond in \cite{vancyc} for mono-germs with $p=n+1$ by counting the  number of spheres in the bouquet of spheres for the disentanglement.  
The definition here uses rational cohomology for multi-germs with $n\leq p$ and hence coincides with Mond's number in his case. 

\item For $p\neq n+1$, Theorem~\ref{dis_cohomology} gives that the $k$th alternating Milnor number, $\mu ^{alt}_k (f,y)$, is just the dimension of the group $\widetilde{H}^{p-(p-n-1)k-1}(\dis (f);\Q) $ for $k\leq d(f)$ and equal to the dimension of 
$\widetilde{H}^{d(f)+1}(\dis (f);\Q) $ for $k= d(f)+1$. 
\end{enumerate}
\end{remarks}

We are now in a position to find the alternating Milnor number and to this end will generalise Marar's formula in \cite{marwark}. 
Generalised formulae were also given in \cite{mm} but, due to a couple of minor errors, these are incorrect as  one can see for the calculation of triple points of a corank 1 map from $(\C ^2,0)$ to $(\C ^3,0)$. The correct formula is given in Example~\ref{correct_examples}. Other generalisations can be found in \cite{genim}.

First we need to investigate the alternating cohomology of an $S_k$-invariant isolated complete intersection singularity. We use this to find a formula in the mono-germ case which is then applied in the multi-germ case.

Let $S_k$ act on $\C ^m\times \C ^k$ by permutation of the last $k$ coordinates. Let $h:\C ^m\times \C ^k \to \C ^r$ define an $S_k$-invariant isolated complete intersection at $0$ given by $S_k$-invariant defining equations.

For an $S_k$-invariant set $Z\subseteq \C^m\times \C ^k$ we denote by $Z^\sigma $ the fixed point set of $Z$ with respect to $\sigma \in S_k$.

Suppose that for all $\sigma \in S_k$ that $h|(\C^m\times \C ^k)^\sigma $ defines an isolated complete intersection singularity. Let $X$ be a Milnor ball (open or closed -- it doesn't matter which in these calculations as the resulting spaces are homotopically equivalent) for $h$ and denote by $\MF $ the (open or closed) Milnor fibre of $h$. Note that $\MF ^\sigma $ is the Milnor fibre of $h|X^\sigma $.

\begin{proposition}
\label{alt_of_pair}
The group $S_k$ acts on the cohomology of the pair $(X,\MF )$ and we have
\[
\dim \Alt _k H^{\dim \MF +1}(X,\MF ;\Q )
=  \frac{1}{k!} 
\left( \sum_{\sigma {\text{ s.t.\ }} \MF ^{\sigma } = \emptyset } \mu (h|X^\sigma )
+ 
(-1)^{\dim \MF  +1} 
\sum_{\sigma {\text{ s.t.\ }} \MF ^{\sigma } \neq  \emptyset } \sign (\sigma )
\right) .
\]
This equality also holds in the case that $\MF =\emptyset $ and if we define $\dim \emptyset =-1$.
\end{proposition}
\begin{proof}
The proof is similar to those in \cite{gomo, hk, wallsym}. Here we give a relative version.

Let $[M]$ denote an element of the ring of all $\Q $-linear representations of $S_k$. For a pair $(A,B)$ of $S_k$-invariant complex analytic spaces one has an equivariant Euler characteristic $\chi _{S_k}(A,B)$ given by 
\[
\chi _{S_k} (A,B)= \sum_{q} (-1)^q [H^q (X,A;\Q )] .
\]
(The following holds in more generality than for complex analytic spaces -- the main property we are using is that complex analytic spaces can be triangulated, see \cite{wallsym}.)

For all $\sigma \in S_k$, $\chi _{S_k}(A,B)(\sigma )$ is equal to the usual topological Euler characteristic of the fixed point pair $(A^\sigma , B^\sigma )$, see \cite{wallsym}.
By the general theory of representations we can show that 
\[
\chi ^{alt} (A,B) = \frac{1}{k!} \sum_{\sigma \in S_k} \sign(\sigma ) \chi (A^\sigma , B^\sigma ).
\]
We can now apply this theory to the pair $(A,B)=(X,\MF )$.

Since $\MF $ is the Milnor fibre of an isolated complete intersection singularity and hence is homotopy equivalent to a wedge of spheres of real dimension $\dim_\C \MF $ and $X$ is contractible, we have 
\begin{eqnarray*}
\chi _{S_k} (X,\MF ) &=& \sum _q (-1)^q[H^q(X,\MF ;\Q ) ] \\
&=& (-1)^{\dim \MF  +1}[ H^{\dim \MF  +1} (X,\MF ;\Q )] .
\end{eqnarray*}
We have
\[
\chi _{S_k} (X,\MF ) (\sigma ) = \chi (X^\sigma ,\MF ^\sigma )
\]
and so 
\[
(-1)^{\dim \MF  +1}[ H^{\dim \MF  +1} (X,\MF ;\Q )](\sigma ) = 
\left\{ 
\begin{array}{ll}
(-1)^{\dim \MF ^\sigma  +1} \mu (h|X^\sigma ), & {\mbox{ for }} \MF ^\sigma \neq \emptyset ,\\
1,  & {\mbox{ for }} \MF ^\sigma = \emptyset .\\
\end{array}
\right.
\]
The latter relative Euler characteristic is $1$ since $X^\sigma $ is contractible to a point.
Thus, 
\[
[ H^{\dim \MF  +1} (X,\MF ;\Q )](\sigma ) = 
\left\{ 
\begin{array}{ll}
(-1)^{\dim \MF  - \dim \MF ^\sigma  } \mu (h|X^\sigma ), & {\mbox{ for }} \MF ^\sigma \neq \emptyset ,\\
(-1)^{\dim \MF  +1},  & {\mbox{ for }} \MF ^\sigma = \emptyset .\\
\end{array}
\right.
\]
Now note that
\begin{eqnarray*}
& &\dim \Alt _k H^{\dim \MF  +1}(X,\MF ;\Q ) \\
&=& \frac{1}{k!} \sum_{\sigma \in S_k} \sign (\sigma )  [ H^{\dim \MF  +1} (X,\MF ;\Q )](\sigma ) \\
&=& \frac{1}{k!} \left( \sum_{\sigma {\text{ s.t.\ }} \MF ^\sigma \neq \emptyset } \sign (\sigma )
(-1)^{\dim \MF  - \dim \MF ^\sigma  } \mu (h|X^\sigma )
+
\sum_{\sigma {\text{ s.t.\ }} \MF ^\sigma \neq \emptyset } \sign (\sigma ) (-1)^{\dim \MF  +1}
\right) .
\end{eqnarray*}
Now, 
\begin{eqnarray*}
\dim \MF  &=& n+k  - \codim (\MF , \C ^{n+k})\\
{\mbox{and }} \dim \MF ^\sigma &=& \dim (\C ^{n+k})^\sigma - \codim  
(\MF ^\sigma , (\C ^{n+k})^\sigma ).
\end{eqnarray*}
It is straightforward to show that the above codimensions are the same (they are defined by the same number of equations) and to show that $\dim (\C ^{n+k})^\sigma = n+ \sum \alpha _i$. Hence, $\dim \MF  - \dim \MF ^\sigma   = k+\sum \alpha _i  $.
We can also calculate that $\sign (\sigma )= (-1)^{k-\sum \alpha _i }$. Thus, 
$\sign (\sigma ) (-1) ^{ \dim \MF  - \dim \MF ^\sigma  } =1$.
\end{proof}

We now generalise Marar's formula from \cite{marwark}.
\begin{theorem}
\label{marar_formula}
Suppose that $f:(\C ^n,\underline{z})\to (\C ^p,0)$, $n<p$, is a corank 1 finitely $\AA $-determined multi-germ and that $\widetilde{f} $ denotes its disentanglement map.

Let $x\in \widetilde{D}^k(f)$ be a singular point
and let $H_x$ be the isotropy group of $x$ under $S_k$; this is isomorphic to $S_j$ for some $j\leq k $ 
(and actually equal to $S_k$ if $f$ is a mono-germ).
 
Then, for $k\leq d(f)$, we define
\[
\mu ^{alt}_k (f)[x]:=
\frac{1}{|H_x|} 
\left( \sum_{\sigma {\text{ s.t.\ }} \widetilde{D}^k(\widetilde{f},\PP _\sigma  ) \neq \emptyset } \mu ( \widetilde{D}^k(f,\PP _\sigma  ),x )
+ 
(-1)^{\dim \widetilde{D}^k(f,\PP _\sigma  ) +1} 
\sum_{\sigma {\text{ s.t.\ }}  \widetilde{D}^k(\widetilde{f},\PP _\sigma )  =   \emptyset } \sign (\sigma )
\right) .
\]
Here, $\mu (X,x)$ denotes the Milnor number at the point $x$ and s.t.\ denotes `such that'.

Then, we have,
\[
\mu ^{alt}_k (f,0) = \sum _{w\in \orbit (x) {\text{ s.t.\ }} \epsilon _k(x)=0} \mu ^{alt}_k (f)[w] ,
\]
where $\orbit (x)$ is the orbit of $x$ considered as a set with no repetitions and recalling that $\epsilon _k:\widetilde{D}^k(f)\to \C ^p$ is given by $\epsilon (x_1,\dots ,x_k)=f(x_1)$.
\end{theorem}
\begin{proof}
That $\mu ^{alt}_k(f)[x]$ calculates the alternating cohomology with respect to the isotropy group of $x$ follows from the proposition above.

To calculate the $S_k$-alternating cohomology for the whole of $\widetilde{D}^k(\widetilde{f})$ we note that if $Z\subseteq X^k$ is $S_k$-invariant, then $Z=\coprod _j \orbit (Z_j)$ for some connected components $Z_j$, and we have 
\[
\Alt _k H^i(Z;\Q )= \oplus _j \Alt _k H^i(\orbit (Z_j); \Q ) .
\]
If $H$ is a subgroup of $S_k$ such that $\sigma (Z_j)=Z_j$ for all $\sigma \in H$, then it is easy to show that
\[
\Alt _k H^i(\orbit (Z_j); \Q ) \cong \Alt _{H} H^i (Z_j ;\Q ) 
\]
where $\Alt _{H} $ means that the alternation is taken over the elements of $H$ rather than $S_k$.

For our situation where $Z_j$ is a connected component of $D^k(\MF )$ we can see that we can calculate
$\Alt _{H} H^i (Z_j ;\Q ) $ using Proposition \ref{alt_of_pair}.

This completes the argument.
\end{proof}

Let us see how the following well-known results 
can be deduced from the above. See \cite{genim} for more general versions.
\begin{example}
\label{correct_examples}
Suppose that $f:(\C ^2,0)\to (\C ^3,0)$ is a corank 1 finitely $\AA $-determined map-germ. Then,
\begin{enumerate}
\item $\mu ^{alt}_2 (f)= \frac{1}{2} \left( \mu (\widetilde{D}^2(f)) + \mu (\widetilde{D}^2(f)|H) \right)$, where $H$ is the fixed point set of the action of $S_2$ on $\C ^2\times \C^2$;
\item (Cf. \cite{r2r3} Prop.\ 3.7.) Suppose that $\widetilde{D}^3(\widetilde{f})\neq \emptyset $. Then $\mu ^{alt}_3 (f)=\frac{1}{6} \left( \mu (\widetilde{D}^3(f)) + 1 \right)$.
\end{enumerate}
\end{example}
\begin{proof}
(i) Here, since $k=2$, we have only two elements in $S_2=\{ id , \sigma \}$, and hence only two partitions $\PP = (1,1)$ and $\PP = (2)$. For $\widetilde{D}^2(f,(1,1))=\widetilde{D}^2(f)$ we have a curve and $\widetilde{D}^2(f,(2))=\mu (\widetilde{D}^2(f)|H) )$ is an isolated point. In both cases we have $\widetilde{D}^2(\widetilde{f}, \PP)\neq \emptyset $ and so $\mu ^{alt}_2 (f)$ has the form stated. 

(ii) Here we have that  $\PP $ is $(1,1,1)$, $(1,2)$, or $(3)$.
The set $\widetilde{D}^3(\widetilde{f},(1,1,1))$ is a finite set of points, each $S_3$-orbit corresponds to a triple point in the image. The sets $\widetilde{D}^k(\widetilde{f} ,\PP )$ for $\PP \neq (1,1,1)$ are all empty. (If they were not, then $\widetilde{f}$ would not be a stable map.) Hence, 
\[
(-1)^{\dim D^3(f)+1}\sum_{\widetilde{D}^k(\widetilde{f} , \PP )\neq \emptyset } (-1) ^{k-\sum \alpha _i} 
=(-1)^1(-3+2)=1.  
\]
Therefore, $\mu ^{alt}_3(f)$ has the described form.
\end{proof}

\begin{example}[Example 2.11 of \cite{hk}]
Suppose that $f:(\C ^3,0)\to (\C ^4,0)$ is a corank 1 finitely $\AA $-determined map-germ and that $\widetilde{D}^3(\widetilde{f})\neq \emptyset $. In this case $\widetilde{D}^3(\widetilde{f})$ is a curve, (i.e., for the partition associated to the identity in $S_3$), the restriction to the simple transpositions $\widetilde{D}^3(\widetilde{f})| H$, where $H$ is the fixed point set of a simple transposition, is finite set of points. The set $\widetilde{D}^3(\widetilde{f}, \PP _\sigma  )$ is empty for the two remaining elements of $S_3$. 

Therefore, we have
\[
\mu ^{alt}_3 (f)=\frac{1}{6} \left( \mu (\widetilde{D}^3(f)) + 3\mu (\widetilde{D}^3(f)| H)  + 2 \right) .
\]
\end{example}

\begin{remark}
Similar formulas are used in Section 2 of \cite{diswe}.  
\end{remark}

\section{Upper semi-continuity of the image Milnor number for corank 1} 
\label{uppersemi}
Let $f:(\C ^n,\underline{z})\to (\C ^p,0)$ be a finitely $\AA $-determined corank $1$ multi-germ, $n<p$, and let
$F:(\C ^n \times \C ,\underline{z}\times  0  )\to (\C ^p\times \C,0\times 0)$, be a one-parameter unfolding of the form where for
a representative of $F$, (also denoted by $F$), 
$F(x,t)=(f_t(x),t)$ with $f_0=f$ and $f_t(x)=0 $ for all $x\in\underline{z}$. Such an unfolding is called {\em{origin-preserving}}.

In the following sections we will be concerned with the maps $f_t$ such that $t$ lies in a contractible open neighbourhood of $0$ in $\{0\} \times \C \subset \C ^p \times \C $. We denote this neighbourhood by $T$.

It should be noted in the following that $d(f_t)$ need not be constant for $t\in T$. (Recall that $d(f)$ is the largest $k$ such that $\widetilde{D}^k(\widetilde{f})\neq \emptyset $.) For example, even for mono-germs, we can have a change. Let $f_t:(\C ,0) \to (\C ^2,0)$ be given by $f_t(x)=(x^2,x^3+tx)$. Then $f_0$ is a cusp and hence $\widetilde{D}^2(\widetilde{f_0})\neq \emptyset $, but $f_t$ is an immersion at the origin in $\C $ for $0\neq t\in T$, and hence $\widetilde{D}^2(\widetilde{f_t})= \emptyset $.

Similarly, if we define (as usual) $s(f_t)$ to be the number of branches of $f_t$, i.e., the cardinality of $f_t^{-1}(0)$, then $s(f_t)$ may not be constant in the family. In other words, there may be a branch of $f_t$ that fuses with another in $f_0$ as $t$ tends to zero. (One can use the cusp example given in preceding paragraph with diffeomorphisms applied to source and target so that for $t\neq 0$ the double point in the image remains at the origin and one of its preimages is the origin in the source.)

It is well known that the Milnor number of an isolated complete intersection is upper semi-continuous; that is, in a family of such singularities, denoted $h_t$,  we have $\mu (h_t)\leq \mu (h_0)$ for all $t$ in some neighbourhood of $0$
(see, for example, page 126 of \cite{looi}). We now show that, under very general conditions, the alternating Milnor numbers have the same property.

\begin{lemma}
\label{mualtlemma}
Suppose that $\widetilde{D}^k(\widetilde{f}_t)\neq \emptyset $ for all $t\in T$. Then $\mu ^{alt}_k$ is upper semi-continuous. That is, $\mu ^{alt}_k (f_t) \leq \mu ^{alt}_k (f_0)$ for all $t$ in some neighbourhood of $0$. 
\end{lemma}
\begin{proof}
By assumption $\widetilde{D}^k(\widetilde{f}_0)$ and $\widetilde{D}^k(\widetilde{f}_t)$ (for $t\neq 0$) are both non-empty. 
Then, simply by considerations of dimension arising from the fact that disentanglement maps are stable, (see Theorem~\ref{mps_for_stable}), we have $\widetilde{D}^k(\widetilde{f}_0, \PP )\neq \emptyset $ if and only if $\widetilde{D}^k(\widetilde{f}_t, \PP )\neq \emptyset $ (where $\PP $ is any partition). Thus, the second summation in Theorem~\ref{marar_formula} is constant in the family. 
The Milnor number terms in Theorem~\ref{marar_formula} are upper semi-continuous and so the first summation in Theorem~\ref{marar_formula} is also upper semi-continuous.

Therefore, $\mu ^{alt}_k$ is upper semi-continuous as it is the sum of upper semi-continuous invariants and a constant.
\end{proof}

\begin{remark}
There is an obstruction to $\mu _I$ being upper semi-continuous because currently it is not known whether the summand $\mu ^{alt}_{d(f_t)+1}$  is upper semi-continuous or not. 
Hence, it would be interesting to clarify the behaviour of $\mu ^{alt}_{d(f_t)+1}$ in general. Fortunately, for a map $f$ with $s(f)\leq d(f)$ we have $\mu ^{alt}_{d(f)+1}=0$ (see Remark~\ref{topmufiszero}). This allows us to prove the following theorem.
\end{remark}

\begin{theorem}
\label{mualt_upper}
Suppose that $s(f_t)\leq d(f_t)$ for all $t\in T$. Then, the Image Milnor number $\mu _I$ is upper semi-continuous.
\end{theorem}
\begin{proof}
As already remarked $s(f_t)\leq d(f_t)$ for all $t\in T$ implies that $\mu ^{alt}_{d(f_t)+1}(f_t)=0$ for all $t\in T$.
Therefore, by definition, the invariant $\mu _I(f_t)$ is a sum of numbers determined solely from the multiple point spaces of the disentanglement map.

If $\widetilde{D}^k(\widetilde{f}_t)$ is non-empty for $t\neq 0$, then $\widetilde{D}^k(\widetilde{f}_0)$ is non-empty. To see this, first let $H$ be a versal unfolding of $f_0$. Then $F$, and hence each $f_t$, can be induced from $H\times \id _\C $ where $\id _\C$ is the identity map on $\C $. Furthermore, $D^k(H\times \id _\C)$ is a manifold by Theorem~\ref{mps_for_stable} and each $D^k(f_t)$ is induced from $D^k(H\times \id _\C)$ by taking a non-singular slice.
Now $\widetilde{D}^k(\widetilde{f}_t)\neq \emptyset$ for $t\neq 0$ implies that $\widetilde{D}^k(f_t)\neq \emptyset $ too. By the Marar-Mond description of multiple point spaces in Corollary~\ref{mps_for_findet} this means that
$\widetilde{D}^k(f_t)$ is a complete intersection (possibly zero-dimensional).

By taking the closure in $D^k(H\times \id _\C )$ of $\widetilde{D}^k(f_t)$ for all $t\neq 0$ we see that $\widetilde{D}^k(f_0)\neq \emptyset $. Purely by analysis of dimension and the Marar-Mond description we can see that this implies that $\widetilde{D}^k(f_0)$ is a complete intersection. This implies that $\widetilde{D}^k(\widetilde{f}_0)\neq \emptyset $. (Note that this uses the complete intersection property. Being non-empty is not enough since for $nk-p(k-1)<0$ the multiple point space $\widetilde{D}^k(f_0)$ could be non-empty but 
$\widetilde{D}^k(\widetilde{f}_0)$ could be empty.)

Therefore, if $\widetilde{D}^k(\widetilde{f}_t)$ for all $t\neq 0$ is non-empty, then, by Lemma~\ref{mualtlemma}, $\mu ^{alt}_k$ is upper semi-continuous.

Next, suppose that $\widetilde{D}^k(\widetilde{f}_0)$ is non-empty and $\widetilde{D}^k(\widetilde{f}_t)$ is empty. Then, 
$\mu ^{alt}_k$ is trivially upper semi-continuous.

Hence, the invariant $\mu _I$ is a sum of upper semi-continuous invariants and so is upper semi-continuous.
\end{proof}

An obvious corollary, but one which is worth stating, is the following.
\begin{corollary}
\label{monouppersemi}
Suppose that $f$ is a mono-germ and $s(f_t)$ is constant. Then $\mu _I$ is upper semi-continuous.
\end{corollary}
\begin{proof}
We have $s(f_t)=1$ for all $t\in T$. Hence, the corollary follows from the fact that $d(f_t)\geq 1$ for all $t\in T$.
\end{proof}

As mentioned earlier, for a map $f$, $\mu _I(f)$ contains the summand $\mu ^{alt}_{d(f)+1}$ which we do not know whether or not is upper semi-continuous. The previous propositions circumvented this problem by forcing this number to be $0$. We now go in a different direction and force it to be a (non-zero) constant.

\begin{theorem}
\label{mualt_upper2}
Suppose that $s(f_t)$ and $d(f_t)$ are constant for all $t\in T$. Then, 
\begin{enumerate}
\item $\mu ^{alt}_k$ is upper semi-continuous for all $k$,
\item $\mu _I$ is upper semi-continuous.
\end{enumerate}
\end{theorem}
\begin{proof}
Part (ii) obviously follows from part (i). For part (i) we know from the proof of Theorem~\ref{mualt_upper} that if $\widetilde{D}^k(\widetilde{f}_0)$ and $\widetilde{D}^k(\widetilde{f}_t)$ are both non-empty, then $\mu ^{alt}_k$ is upper semi-continuous. Hence we know that $\mu ^{alt}_k$ is upper semi-continuous for $k\leq d(f_t)$. But then, as $s(f_t)$ and $d(f_t)$ is constant in the unfolding, by definition $\mu ^{alt}_{d(f_t)+1}(f_t)$ is constant too, and hence $\mu ^{alt}_{d(f_t)+1}$ is upper semi-continuous.
\end{proof}

Consider now what happens when $\mu _I$ is constant in a family.
\begin{theorem}
\label{constancytheorem}
Suppose that $f:(\C ^n,\underline{z})\to (\C ^p,0)$, $n<p$, is a finitely $\AA $-determined multi-germ and $F$ is an origin-preserving unfolding with $s(f_t)\leq d(f_t)$ for all $t\in T$ or both $s(f_t)$ and  $d(t_t)$ are constant for all $t\in T$. 
Then,
\begin{eqnarray*}
& & \mu _I(f_t,0) {\text{ is constant for all }} t\in T, \\
&\iff & \mu _k^{alt} (f_t,0) {\text { is constant for all }}t\in T, {\text{ and all }} k, \\
&\iff & \mu (\widetilde{D}^k(f_t,\PP _\sigma ),w ) {\text { is constant for all }}t\in T, {\text{ all }} k {\text{ and }}
\PP _\sigma , {\text{ and }} w\in \epsilon _k^{-1}(0) .
\end{eqnarray*}  
\end{theorem}
\begin{proof}
Here we use the idea that if a (positive) sum of upper semi-continuous invariants is constant in a family, then each summand is constant in the family.

If $\mu _I(f_t)$ is constant along $T$, then by Theorem~\ref{mualt_upper} or Theorem~\ref{mualt_upper2} as the case may be, we have
$\mu _k^{alt} (f_t)$ constant along $T$ for all $k$. The converse follows from the definition.

If $\mu _k^{alt}(f_t)$ is constant along $T$, then by Theorem~\ref{marar_formula} the $\mu (\widetilde{D}^k(f_t),\PP _\sigma ) $ are constant too as these latter invariants are upper semi-continuous. 
The converse statement also follows obviously from Theorem~\ref{marar_formula}.
\end{proof}

\begin{corollary}
Suppose that $f:(\C ^n,0)\to (\C ^p,0)$, $n<p$, is a finitely $\AA $-determined mono-germ and $F$ is an origin-preserving unfolding
with $s(f_t)$ constant for all $t\in T$. 
Then,
\begin{eqnarray*}
& & \mu _I(f_t,0) {\text { is constant for all }} t\in T, \\
&\iff & \mu _k^{alt} (f_t,0) {\text { is constant for all }}t\in T, {\text{ and all }} k, \\
&\iff & \mu (\widetilde{D}^k(f_t,\PP _\sigma ),w ) {\text { is constant for all }}t\in T, {\text{ all }} k {\text{ and }}
\PP _\sigma , {\text{ and }} w\in \epsilon _k^{-1}(0) .
\end{eqnarray*}  
\end{corollary}
\begin{proof}
As $f$ is a mono-germ $s(f_t)=1\leq d(f_t)$ for all $t$.
\end{proof}

%
%
%
\section{$\mu _I$-constant and excellent unfoldings} 
\label{unfoldings}
We now discuss excellent unfoldings and indicate their significance -- that is, make precise what the definition is intended to achieve. For an unfolding of a mapping we would like members of the family to be equivalent in some sense. This is the study of equisingularity. 
One example is Whitney equisingularity. Here we Whitney stratify the map so that we can apply Thom's Isotopy Lemma and hence deduce that the members of the family are topologically equivalent. This application can only be achieved if the parameter axes (in source and target) of the unfolding are the only one-dimensional strata. 
The conditions for an unfolding to be excellent imply that the parameter axes are the only one-dimensional stratum.

Let us consider stratifications of unfoldings of finitely $\AA $-determined map-germs.
In order to do this we first define stratification by stable type.
A good reference for this is Section 2.5 of \cite{dupwall}.

Let $G:(\C ^n,\underline{z} )\to (\C ^p,0)$ be a stable map with $n<p$. 
There exist open sets $U\subseteq \C ^n $ and $W \subseteq \C ^p$ such that 
$G^{-1}(W) = U $ and $G:U\to W$ is a representative of $G$.
We can partition $G(U)$ by stable type. That is, $y_1$ and $y_2$ in $\C ^p$
have the same stable type if $G_1:(\C ^n , G_1^{-1}(y_1) ) \to (\C ^p,y_1)$ and
$G_2:(\C ^n , G_2^{-1}(y_2) ) \to (\C ^p,y_2)$ are $\AA $-equivalent. These sets are
complex analytic manifolds.
We can take strata in $U\subseteq \C ^n$ by taking the partition
$G^{-1}(S) $ where $S$ is a stratum in the image.

It is possible to show (\cite{polar} Lemma 7.2 and \cite{dupwall} Section 2.5) that if $G$ is in the nice dimensions or a corank 1 map-germ, then this stratification of source and target by stable type is Whitney regular and any Whitney stratification of $G$ is a refinement of this, i.e., this stratification is canonical.

Now, for the stratification of the unfolding of a finitely $\AA $-determined multi-germ we can stratify the stable parts by stable type. Since the original map was finitely $\AA $-determined this means that in a neighbourhood of the origin of $\C ^p \times \C $ only a collection of curves remains to be stratified after stratification by stable type. For an excellent unfolding (defined later) we would like the parameter axes in $\C ^n \times \C $ and $\C ^p \times \C $ to be the only one-dimensional strata.

Ignoring the parameter axes we can get one-dimensional strata occurring in the stratification in the following ways, see \cite{polar}.
\begin{enumerate}
\item The set $f_t^{-1}(\{ 0\} )$ contains more than $\underline{z} $. That is, there exists an extra branch or branches in $f_t$ passing through $0\in \C ^p $ which `fuse' with other branches at $t=0$.
\item There is a curve of unstable points in $\C ^p\times \C $, other than $\{ 0\} \times \C $, which passes through $0$.
\item There exists a curve of points arising from $f_t$ having a stable type with a zero-dimensional stratum.
\end{enumerate}
Obviously, we can control (i) by requiring that $s(f_t)$ is constant.
To investigate (ii) we make the following definition.
\begin{definition}
The {\em{instability locus}} of a map $g$ is the germ of the set of points $y\in \C ^p$ such that
$g:(\C ^n,g^{-1}(y))\to (\C ^p,y)$ is not stable.
\end{definition}

For (iii) we make the following definition.
\begin{definition}(See \cite{polar}.)
A stable type is called {\em{$0$-stable}} if the stratification by stable type has a $0$-dimensional stratum.
\end{definition}

\begin{examples}
The Whitney cross-cap $(x,y)\mapsto (x, y^2, xy)$ is $0$-stable. The multi-germ from $(\C ^2 , \{x_1, x_2 ,x_3\} )$ to $(\C ^3,0)$
giving an ordinary triple point is $0$-stable.
\end{examples}
By counting the $0$-stable singularities that appear in a stable perturbation of a map from $\C ^2$ to $\C ^3$ with an isolated instability
Mond was able, in \cite{r2r3}, to produce useful invariants.

Our interest here is that it is easy to see for corank $1$ multi-germs that a $0$-stable type corresponds to $\widetilde{D}^k(f,\PP )$ for some $k$ and $\PP $, where $\widetilde{D}^k(f,\PP )$ is zero-dimensional.
Conversely, if $f$ is stable and $\widetilde{D}^k(f,\PP )$ is zero-dimensional, then this corresponds to a $0$-stable singularity in the target.

\begin{definition}
In an unfolding, if 
there is a sequence of $0$-stables
converging to the origin in $\C ^p \times \C$, then there is a curve of $0$-stables in the image of $F$.
If no such curve exists, then we say the {\em{$0$-stables are constant in the family}}.
\end{definition}

We now come to a crucial definition of this paper -- excellent unfolding -- introduced by Gaffney in the study of equisingularity of mappings in \cite{polar}. Here we restrict to the case of corank $1$ maps with $n<p$ which simplifies the definition and we make the simple generalisation to multi-germs. 
\begin{definition}
\label{exc_def}
Suppose that $f:(\C ^n , \underline{z} ) \to (\C ^p ,0)$, $n<p$, is a corank 1 finitely $\AA $-determined multi-germ and that $F$ is
an origin-preserving one-parameter unfolding such that $F|U \to W $ is proper and finite-to-one.

We call $F$ a {\em{good unfolding}} if there exists a contractible neighbourhood $T$ of $\{0\} \times \C \subset \C ^p \times \C$, such that all the following hold.
\begin{enumerate}
\item $F^{-1}(W)=U$.
\item $F(U  \backslash (\underline{z} \times \C ) )= W\backslash T $, i.e., $s(f_t)$ is constant.
\item The locus of instability is contained in $T$.
\end{enumerate}
We call $F$ an {\em{excellent unfolding}} if in addition we have the following.
\begin{enumerate}
\setcounter{enumi}{3}
\item The $0$-stables are constant along $T$.
\end{enumerate}
\end{definition}

\begin{remark}
The first condition is easiest to achieve as we can just restrict our domain to the preimage of $W$. It is really the last three conditions that need to be checked. Algebraic conditions for an unfolding to be good are given, using work of Damon, in Proposition~2.3 of \cite{polar}.
\end{remark}
Constancy of the image Milnor number in a family is sufficient to imply that the $0$-stables are constant.
\begin{theorem}
\label{0stablesconstant}
Suppose that $f:(\C ^n , \underline{z} ) \to (\C ^p ,0)$, $n<p$, is a corank 1 finitely $\AA $-determined multi-germ and that $F$ is an origin-preserving one-parameter unfolding such that $F|U \to W $ is proper and finite-to-one. Suppose further that
$s(f_t)\leq d(f_t)$ for all $t\in T$ or both $s(f_t)$ and $d(f_t)$ are constant for all $t\in T$.

Then, $\mu _I(f_t)$ is constant for all $t\in T $ implies that the $0$-stables are constant.
\end{theorem}
\begin{proof}
Consider the multiple point spaces for $F$ and in particular the points $(x_1, \dots , x_k )$ where 
$x_i =(z _i, t)$ with $z_i \in \underline{z}$ and $t\in \C $. We will get a family 
of singularities along these parameter axes in $D^k(f)$, which we will denote by $T_k$. Note that for multi-germs
these manifolds may have many connected components.
We will consider the constancy of Milnor numbers along these axes.

A $0$-stable corresponds to a $0$-dimensional $\widetilde{D}^k(f,\PP )$ for some $k$ and $\PP $.
If there is a curve of $0$-stables that does not lie in $T$, then this means that the number of
points, i.e., the Milnor number plus one, must jump at $t=0$. This is not possible as $\mu _I(f_t)$ is constant along $T$. This is because if $\mu _I(f,t)$ is constant along $T$, then the corresponding Milnor numbers of $\widetilde{D}^k(f_t, \PP _\sigma )$ will be constant along $T_k$ by Theorem \ref{constancytheorem}.
In particular the Milnor numbers of all the $0$-dimensional singularities will be constant.
\end{proof}

\begin{corollary}
Suppose that $f:(\C ^n , 0 ) \to (\C ^p ,0)$, $n<p$, is a corank 1 finitely $\AA $-determined multi-germ and that $F$ is an origin-preserving one-parameter unfolding such that $F|U \to W $ is proper and finite-to-one and $s(f_t)$ is constant. 

Then, $\mu _I(f_t)$ constant for all $t\in T $ implies that the $0$-stables are constant.
\end{corollary}
\begin{proof}
Again, $s(f_t)=1\leq d(f_t)$ and so the corollary follows from the theorem.
\end{proof}
The next corollary also follows simply from the theorem but is worth remarking.
\begin{corollary}
\label{muiconst+good}
For $F$ as in the theorem, if $F$ is good and $\mu _I(f_t)$ is constant, then  $F$ is excellent.
\end{corollary}

%
%

The invariant $\mu _I$ is important in the study of equisingularity of mappings. For example, see \cite{diswe} and \cite{newequi}. The above corollaries are of great interest because they allow us to relate the excellence of an unfolding to $\mu _I$ as well, thus advancing the conjecture that, in equisingularity theory, this invariant and (others like it) provide sufficient conditions for equisingularity.

\begin{theorem}
\label{main_thm}
Let $f:(\C ^n , \underline{z} ) \to (\C ^p ,0)$, $n<p$, be a corank 1 finitely $\AA $-determined multi-germ and $F$ be an origin-preserving one-parameter unfolding such that $F|U \to W $ is proper and finite-to-one.

Suppose that $f$ has a one-parameter stable unfolding and that $s(f)\leq d(t_t)$ for all $t$ or both $s(f_t)$ and $d(f_t)$ are constant for all $t$.

Then $\mu _I$ constant implies that the instability locus of $F$ is $T$.
\end{theorem}
\begin{proof}

Since $f$ has a one-parameter stable unfolding, say $G$, we can assume this holds for all nearby $t$. Therefore $\widetilde{D}^k(f)$ is a hypersurface in $D^k(G)$ with a similar result for all $t$. We can therefore assume that  $\widetilde{D}^k(f_t)$ is a family of hypersurfaces.

Now suppose that the instability locus was not $T$ and there existed another curve of unstable points passing through origin. Then, by Theorem \ref{mps_for_findet} there exists a curve of points in some multiple point space that are singular points not equal to a $k$-tuple from $\underline{z}$. 

In this case, since the Milnor number for hypersurface singularities is additive, the Milnor numbers in this family will jump at $t=0$. But as $\mu _I$ is constant we can see from Theorem~\ref{constancytheorem} that this is not possible.
\end{proof}

\begin{corollary}
Let $f:(\C ^n , \underline{z} ) \to (\C ^p ,0)$, $n<p$, be a corank 1 finitely $\AA $-determined multi-germ and $F$ be an origin-preserving one-parameter unfolding such that $F|U \to W $ is proper and finite-to-one.

Suppose that $f$ has a one-parameter stable unfolding and that both $s(f_t)$ and $d(f_t)$ are constant for all $t$.

Then $\mu _I$ constant implies that $F$ is excellent.
\end{corollary}
\begin{proof}
Obviously we can choose $U$ so that condition (i) of an excellent unfolding holds. As $s(f_t)$ is constant we have condition (ii). As both $s(f_t)$ and $d(f_t)$ are constant then by Theorem~\ref{0stablesconstant} the $0$-stables are constant, i.e., condition (iii) holds. Similarly Theorem~\ref{main_thm} gives condition (iv).
\end{proof}

From the preceding results we can see that we can control two conditions from the definition of excellent unfolding by using $\mu _I$.
The main element at the heart of the proofs is that
we know the cohomology description of the multiple point spaces (Theorem~\ref{dis_cohomology} and Lemma~\ref{E1sparse})
and that the invariants are upper semi-continuous (from Section~\ref{uppersemi}). Hence, if we have non-corank $1$ situations where these hold for the multiple point spaces, then we can make statements similar to that of Theorem~\ref{0stablesconstant} and, when we have additivity of Milnor number, Theorem~\ref{main_thm}.

For example, Buchweitz and Greuel, in \cite{bg}, show that one can define an additive upper semi-continuous Milnor number for certain families of curves. This means that if our multiple point spaces are curves, we can state the following theorem. 
\begin{theorem}
\label{closingtheorem}
Let $f:(\C ^n,\underline{z})\to (\C^ p,0)$, $n<p$, be a finitely $\AA $-determined map-germ such that one of the following holds:
\begin{enumerate}
\item $(n,p)=(n,2n-1)$, for $n\geq 2$, so in particular this includes $(n,p)=(2,3)$. 
\item $(n,p)=(n,2n)$, $n\geq 1$.
\end{enumerate}
Suppose further that $F$ is an origin-preserving one-parameter unfolding of $f$ with $s(f_t)\leq d(f_t)$ for all $t\in T$ or both $s(f_t)$ and $d(f_t)$ are constant.

Then, $\mu _I(f_t)$ constant for all $t$ implies that the instability locus of $F$ is $T$ and the $0$-stables are constant for all $t\in T$.
\end{theorem}
\begin{proof}
(i) In this case $\widetilde{D}^2(f)$ is a curve and the stable singularities will be corank $1$. Thus, 
$\widetilde{D}^2(\widetilde{f})$ will be a smoothing of this curve. Furthermore, as in the case of mono-germs shown in \cite{bhr}, it will lie in flat family as required by \cite{bg} for upper semi-continuity to hold.

The restriction of $\widetilde{D}^2(f)$ is a zero dimensional Cohen-Macaulay space and in a smoothing the number of pairs of points in $\widetilde{D}^2(\widetilde{f})$ will be the number of Whitney umbrellas in the image of the stabilisation. These can be counted as the number of points can be calculated from the degree of $\widetilde{D}^2(f)$. 
Thus, $\mu _2^{alt}$ behaves in the same way as for a corank $1$ map. 

If $n>2$, then $\widetilde{D}^3(\widetilde{f})$ is empty. If $n=2$, then it is zero-dimensional and $\mu _3^{alt}$ behaves in the same way as the corank $1$ case.   

(ii) Only $\widetilde{D}^2(\widetilde{f})$ is non-empty and it is zero-dimensional. In a stabilisation, the only singularities are ordinary intersection of two sheets and $\mu _2^{alt}(f)$ merely counts these.
\end{proof}

A similar result has been found in \cite{ballperez}. The authors there consider the case of finitely $\AA $-determined mono-germs $f:(\C ^n,0)\to (\C ^{2n-1},0)$ for $n\geq 3$, (any corank). They consider the double point set in the target, denote this by $\mathcal{D} (f)$, and they show that if $\mathcal{D} (f_t)$ is constant in an unfolding, then the unfolding is excellent.
Using techniques similar to those in \cite{diswe} it is possible to show that $\mathcal{D} (f_t)$ constant is equivalent to $\mu _I(f_t)$ constant, and hence we recover their result.
Precise details and generalisations of these assertions will be given in a forthcoming paper.

\section{Concluding remarks}

\begin{remark}
This paper has shown that $\mu _I$ can be used in certain situations to control conditions (ii) and (iii) for an unfolding to be excellent. 
In some examples it can be seen that it also controls the second condition (i.e., $F(U  \backslash (\underline{z} \times \C ) )= W\backslash T $). For example, in \cite{polar} Theorem~8.7 the condition is controlled by $0$-stables and the singularity structure of the multiple point spaces. Both these can be controlled by $\mu _I$. Therefore, it would be interesting to have a general result of the form `$\mu _I$ constant implies condition (ii) of excellent unfolding.'

In fact, from this and using a number of theorems from the previous section as evidence, it seems that  $\mu _I$ should completely control excellence of unfoldings. It seems natural to conjecture the following:
\begin{conjecture}
Suppose that $f:(\C ^n,\underline{z})\to (\C ^p,0)$, $n<p$, is a corank $1$ complex analytic multi-germ with an origin-preserving one-parameter unfolding. Then $\mu _I$ constant implies that the unfolding is excellent. 
\end{conjecture}
\end{remark}

\begin{remark}
Proving the above conjecture will probably involve clarifying how $\mu ^{alt}_{d(f)+1}$ behaves under deformation so that we have very general results on the upper semi-continuity of $\mu _I$. A second reason is that $\mu ^{alt}_{d(f)+1}$ is connected with condition (ii) of excellent unfolding. To see this suppose that $d(f_t)$ is constant in the family. Then, $\mu _I(f_t)$ constant implies that $\mu _{d(f_t)+1}^{alt}(f_t)$ is constant. Because of the definition, for $s(f_t)>d(f_t)$, it is natural to conjecture that this gives $s(f_t)$ is constant. 
This latter implies that condition (ii) of excellent holds. 
\end{remark}

\begin{remark}
Outside the area of equisingularity, as the structure of $\mu _I$ for multi-germs is detailed in terms of $\mu ^{alt}_k$ it may be possible to use the structure to compare $\mu _I$ and the $\AA _e$-codimension of map-germs. Consider the following.

For an isolated complete intersection singularity $h$ we have $\mu (h)=0$ if and only if $h$ is non-singular (i.e., $\KK _e$-codimension$(h)=0$), and $\mu (h)=1$ if and only if $\KK _e$-codimension$(h)=1$. It seems reasonable to prove that $\mu _I(f)=0$ if and only if $f$ is stable (i.e., $\AA _e$-codimension$(f)=0$) and $\mu _I(f)=1$ if and only if $f$ has $\AA _e$-codimension $1$. This latter may be related to having one multiple point space being a disjoint union of quadratic hypersurface singularities and those below it are non-singular, whilst the remaining are just the origin. These results could be useful in showing that $\mu _I$ constant in a family implies that the instability locus of the family is $T$.
\end{remark}

\begin{remark}
Moving beyond low $\AA _e$-codimension, for an isolated complete intersection singularity of dimension greater than zero it is well-known that $\KK _e$-codimension$(h)\leq \mu (h)$ with equality if $h$ is quasi-homogeneous. It is conjectured by Mond \cite{vancyc} that for $p=n+1$, we have $\mu _I (f)\leq \AA _e$-codimension$(f)$ with equality if $f$ is quasihomogeneous. 
Now, $\mu _I(f)$ is just a sum of $\mu _k^{alt}(f)$ and $\mu ^{alt}_k(f)$ is (for almost all $k$) the alternating part of the homology of $\widetilde{D}^k(f)$. It is possible to show that, in analogy with $\mu (h)\leq \KK _e$-codimension$(h)$, that $\mu ^{alt}_k(f)$ is less than or equal to the dimension of the symmetric part of the $\KK _e$-normal space of $\widetilde{D}^k(f)$. Hence, if one could show that the sum of the symmetric parts gave $\AA _e$-codimension$(f)$ -- which it does in many examples (see also \cite{folaug}) - then it would be possible to prove a general $\mu _I(f)\leq \AA _e$-codimension$(f)$ result. Admittedly, this would require no zero-dimensional multiple point spaces (since $\KK _e$-codimension $\leq \mu$ does not hold in this case) but this still includes a large class of maps.
\end{remark}


\begin{thebibliography}{1234}

\bibitem[1]{bg} R.-O.~Buchweitz, G.-M.~Greuel, The Milnor number and deformations of complex curve singularities, {\em{Invent. Math.}} 58 (1980), no. 3, 241-281.

\bibitem[2]{bhr} R.\ Callejas-Bedregal, K.\ Houston, M.A.S.\ Ruas, Topological triviality of families of singular surfaces, Preprint 2006.

\bibitem[3]{highmult} J.\ Damon, Higher multiplicities and almost free divisors and complete 
intersections,  {\em{Memoirs of AMS}}, Vol. 123, no.\ 589 (1996).

\bibitem[4]{dm} J.\ Damon and D.\ Mond, $\AA $-codimension and
vanishing topology of discriminants, {\em{Invent.\ Math.}}, 106 (1991),
217-242.

\bibitem[5]{dupwall} A.\ du Plessis, C.T.C.\ Wall, {\em{The geometry of topological stability}}, London Mathematical Society Monographs. New Series, 9. Oxford Science Publications. The Clarendon Press, Oxford University Press, New York, 1995.

\bibitem[6]{gaffneymps} T.\ Gaffney, Multiple points and associated
ramification loci, in {\em{Proceedings of Symposia in Pure Mathematics 40:I}},
Peter Orlik (Ed.), American Mathematical Society, Providence 1983,
pp.\ 429-437.

\bibitem[7]{polar} T.\ Gaffney, Polar multiplicities and equisingularity
of map germs, {\em{Topology}}, 32 (1993), 185-223. 

\bibitem[8]{go}  V.V. Goryunov, Semi-simplicial resolutions and homology
of images and discriminants of mappings, {\em{Proc.\ London Math.\ Soc.}}
70 (1995), 363-385.

\bibitem[9]{gomo} V.V.\ Goryunov, D.\ Mond, Vanishing cohomology of
singularities of mappings, {\em{Compositio Math.}} 89 (1993), 45-80.

\bibitem[10]{loctop} K.\ Houston, Local Topology of Images of Finite Complex Analytic Maps, {\em{Topology}}, Vol 36, (1997), 1077-1121.

\bibitem[11]{folaug} K.\ Houston, 
On the Singularities of Folding Maps and Augmentations, {\em{Mathematica Scandinavica}}, (1998), 82, 191-206.

\bibitem[12]{genim} K.\ Houston, Generalised Discriminant and Image Milnor Numbers, {\em{Glasgow Mathematical Journal}}, (2001) 43, 165-175. 

\bibitem[13]{diswe} K.\ Houston, Disentanglements and Whitney Equisingularity,   {\em{Houston Journal of Mathematics}}, 33 no. 3, (2007), 663-681.

\bibitem[14]{icss} K.\ Houston, A General Image Computing Spectral Sequence, {\em{Singularity Theory}}, ed., Denis Ch\'eniot, Nicolas Dutertre, Claudio Murolo, David Trotman and Anne Pichon, World Scientific Publishing, Hackensack, NJ, 2007, 651-675.

\bibitem[15]{newequi} K.\ Houston, Equisingularity of families of hypersurfaces and applications to mappings. Preprint 2007.
Available from http://www.maths.leeds.ac.uk/$\sim $khouston/papers.html

\bibitem[16]{hk} K.~Houston, N.~Kirk, On the classification and geometry of corank $1$ map-germs from three-space to four-space, in {\em{Singularity Theory (Liverpool, 1996)}}, 325-351, London Mathematical Society Lecture Note Series, 263, Cambridge University Press, Cambridge, 1999. 

\bibitem[17]{looi} E.J.N.~Looijenga, {\em{Isolated singular points on complete intersections}}, London Mathematical Society Lecture Note Series, 77, Cambridge University Press, Cambridge, 1984.

\bibitem[18]{mapfib} W.L.~Marar, Mapping fibrations, {\em{Manuscripta Math.}} 80
(1993), 273-281.

\bibitem[19]{marwark} W.L.~Marar, The Euler characteristic of the disentanglement of the image of a corank $1$ map germ, in {\em{Singularity Theory and its Applications}}, SLNM 1462, D.\ Mond,
J. Montaldi (Eds.),  Springer Verlag Berlin, 1991, 212-220.

\bibitem[20]{mm} W.L.~Marar and D.\ Mond, Multiple point schemes for 
corank 1 maps, {\em{J.\ London Math.\ Soc.}} (2) 39 (1989), 553-567. 

\bibitem[21]{r2r3} D.\ Mond, Some remarks on the geometry and classification of germs of maps from surfaces to 3-space, {\em{Topology}}, 26, 3 (1987), 361-383. 

\bibitem[22]{vancyc} D.\ Mond, Vanishing cycles for analytic maps,
in {\em{Singularity Theory and its Applications}}, SLNM 1462, D.\ Mond,
J. Montaldi (Eds.), 
Springer Verlag Berlin, 1991, 221-234.

\bibitem[23]{morin} Bernard Morin,  Formes canoniques des singularit\'es d'une application diff\'erentiable, {\em{C.\ R.\ Acad.\ Sci.\ Paris}} 260 (1965), 5662-5665.

\bibitem[24]{ballperez} V.H.~Jorge P\'erez and J.J.~Nu\~{n}o Ballesteros, Finite determinacy and Whitney equisingularity of map germs from $\C ^n$ to $\C ^{2n-1}$, Preprint 2006.

\bibitem[25]{wallsym} C.T.C. Wall, A note on symmetry of singularities, {\em{Bull. London Math. Soc.}} 12 (1980), no. 3, 169-175.

\bibitem[26]{wall} C.T.C. Wall, Finite determinacy of smooth map-germs,
{\em{Bull.\ Lond.\ Math.\ Soc.}}, 13 (1981), 481-539.

\end{thebibliography}
\end{document}